\documentclass[final,leqno]{siamltex}
\usepackage[english]{babel}
\usepackage{emptypage}
\usepackage{environ,tikz}
\usepackage{envmath}
\usepackage{pgfplots}
\usepgfplotslibrary{polar}
     \usepgflibrary{shapes.geometric}
     \usetikzlibrary{calc}
     \usetikzlibrary{automata,positioning,arrows}

\pgfplotsset{compat=1.3}
\usetikzlibrary{arrows, decorations.markings}
\usetikzlibrary{decorations.text}

\usetikzlibrary{plotmarks}

\usepackage{amsmath,amsfonts,amssymb}%
\usepackage{bbm,enumerate,stackrel,graphicx}
\usepackage{hyperref}
\usepackage{hypcap}
\usepackage{paralist}
\usepackage{showkeys}

\newtheorem{lemm}[theorem]{Lemma}
\newtheorem{defim}[theorem]{Definition}
\newtheorem{propm}[theorem]{Proposition}

\newtheorem{remkm}[theorem]{Remark}

\newtheorem{hyp}[theorem]{Assumptions}
\usepackage{subfigure}
\newtheorem{thmm}[theorem]{Theorem}

\newtheorem{exp et remarque}[theorem]{Exemples et remarques}

\newcommand{\qed}{{ \hfill
    {\unskip\kern 6pt\penalty 500 \raise -2pt\hbox{\vrule\vbox to 6pt{\hrule width 6pt
          \vfill\hrule}\vrule} \par}   }}

\DeclareMathOperator*{\argmax}{argmax}

\begin{document}
\title{Blocking of Bistable Front Propagation: Application to Wolbachia Population Replacement
}
\author{%
  Saoussen Latrach\thanks{Université Sorbonne Paris Nord, Laboratoire Analyse, Géométrie et Applications, LAGA,F-93430 Villetaneuse, France\\
  Emails: latrechsawssen@gmail.com, vauchelet@math.univ-paris13.fr, zaag@math.univ-paris13.fr}%
  \and
  Nicolas Vauchelet\footnotemark[1]%
  \and
  Hatem Zaag\footnotemark[1]%
}

\maketitle

\begin{abstract}
    In order to control epidemics of mosquito-borne diseases such as dengue, for which there is still no vaccine available, several strategies aim to directly target the mosquito population. One technique involves introducing a symbiotic bacterium called Wolbachia into the population, which blocks the transmission of pathogens. A mathematical model of the introduction of this bacterium and the replacement of the mosquito population with a population carrying this bacterium has been developed, proving that in a homogeneous environment, there is spatial propagation of the replacement. The objective of this paper is to study the blocking of this propagation in cases where the environment is not homogeneous. In particular, we are considering scenarios where the environmental carrying capacity varies significantly across different regions of the spatial domain.

To achieve this objective, we first reduce the original 2×2 reaction-diffusion system to a single closed equation, providing a simplified framework for studying the dynamics of Wolbachia propagation. We then investigate the blocking of the propagation in the population replacement process, focusing on the effects of spatial variations in the environmental carrying capacity. Numerical simulations have been performed to validate the theoretical results.

\end{abstract}
\textbf{Keywords :}  Reaction-diffusion system, travelling waves, wave blocking, Wolbachia bacterium, vector-borne diseases.\\

\section{Introduction}
The fight against dengue, a global scourge (see \cite{bhatt2013global}), and other arboviruses has mobilized significant efforts within the scientific community. The development of innovative techniques for vector control has become a particularly dynamic research area. Among these approaches, the use of Wolbachia in Aedes mosquitoes to control the spread of diseases (see \cite{walker2011w}; \cite{alphey2013genetic}) has gained increasing interest. This endosymbiotic bacterium, transmitted maternally, causes cytoplasmic incompatibility: crosses between infected males and uninfected females lead to infertility, while blocking viral replication in mosquito hosts. Artificial infection of mosquitoes with Wolbachia is feasible under laboratory conditions, and vertical transmission allows for rapid and mass reproduction of an infected colony.
The first mathematical models of this method include the work of Barton et al. \cite{barton2011spatial}, Hughes et al. \cite{hughes2013modelling}, and Hancock et al. \cite{hancock2011strategies}, who contributed to a better understanding of the dynamics of this infection and its implications for populations.
The impact of spatial dimension in vector control processes has also attracted attention from many researchers. Several studies have explored the conditions under which mathematical models incorporating spatial variations can predict the blocking of a biological invasion, such as the one resulting from the release of Wolbachia-infected mosquitoes in vector control programs.
The observation of biological invasions and the mechanisms used to stop them is a well-studied topic in the literature. For example, the experimental study conducted by Atyame et al. \cite{atyame2015stable} showed that stable coexistence between several natural Wolbachia strains could be observed in a population of Culex pipiens. Several hypotheses have been proposed to explain this stability. Nadin et al. \cite{nadin2018hindrances}, in a simplified mathematical model, support some of the analyses conducted in these studies. In particular, "differentiated adaptation" alone is insufficient to explain this blocking, but a marked population gradient could play a crucial role. These studies thus allow quantification of the necessary intensity of this gradient, which could either confirm or disprove this hypothesis.
Our own results in this paper provide further support to these analyses. We show that the environmental capacity $ K $ plays a crucial role in propagation. When the environmental capacity spatial variations
are small, blocking propagation is impossible. However, when these variations are sufficiently pronounced, they can halt the invasion. This result highlights the importance of environmental gradients in controlling biological invasions and opens new perspectives for explaining experimental observations.\\ Although long-term field experiments remain limited, artificial releases of Wolbachia-infected mosquitoes \cite{hoffmann2011successful, nguyen2015field} have revealed phenomena such as stable fronts and blockages \cite{hoffmann2014stability, yeap2011dynamics}. These phenomena have been modeled in several studies, notably using reaction-diffusion equations \cite{barton2011spatial, chan2013modelling, Bliman} and models accounting for habitat heterogeneity \cite{hancock2012modelling}.
The reaction-diffusion equations, first introduced in \cite{fisher1937wave} and \cite{kolmogorov1937etude}, provide a framework for mathematically describing biological invasions. These models, in the form:
\begin{equation}\label{eq-0} 
\partial_t u - D \partial_{xx} u = f(u), 
\end{equation}
where $ t $ represents time, $ x \in \mathbb{R} $ denotes the spatial variables, and $ u(t, x) $ is the density of an allele at a given time and position, are widely used in population dynamics. This model is particularly appreciated for its ability to describe spatial spread, especially through "traveling wave" solutions that are time-invariant and propagate at a constant speed. These waves retain a constant shape and allow studying the effects of invasion dynamics under specific conditions.

Blocking propagation is a well-documented phenomenon in biological contexts, as observed for genetic traits and infections in heterogeneous populations. Previous work has shown that spatial heterogeneities in population density, or even the Allee effect, can limit the spread of a trait or infection (see \cite{barton1986effects}). This mechanism, often referred to as "pinning" or blocking, can be explained by spatial variations in the environment that hinder invasion (see \cite{keitt2001allee}). The bistable behavior of reaction-diffusion systems is a key phenomenon in studying such invasions. Bistability refers to the coexistence of two stable states, typically characterized by two distinct equilibrium solutions for $ u $. In the context of biological invasions, this implies that both invasion and non-invasion states can persist, depending on initial conditions and system parameters.
Our approach is based on the analysis of a modified equation in one spatial dimension that accounts for spatial variations:
\begin{equation}
\partial_{t}p - D\partial_{xx} p - 2D\frac{K^{'}}{K}    \partial_{x} p = f(p) + \frac{K^{''}}{K} g(p).
\end{equation}
This formulation explicitly integrates spatial heterogeneity through the terms $ \frac{K{'}}{K} $ and $ \frac{K{''}}{K} $, a determining factor in blocking propagation. Mathematically, our study relies on demonstrating the existence of stationary solutions for this equation, using the shooting method, a commonly employed approach for solving partial differential equations. In particular, we show that when $ \frac{K^{'}}{K} $ and $ \frac{K^{''}}{K} $ are of small amplitude, propagation is not blocked. On the other hand, if these terms are sufficiently large, propagation is blocked.
To simplify the analysis, we assume that the diffusion coefficient is normalized, meaning that $D = 1$ and will maintain this assumption throughout the remainder of the work.\\
In this paper, we first introduce the mathematical model in Section \ref{section2} and present the main results of our study. Section \ref{section3} is devoted to model reduction. Section \ref{section4} contains the preliminary results and the proofs of the main results. In Section \ref{section5}, we provide numerical simulations to illustrate the phenomena of propagation and blocking. Finally, Section \ref{section6} concludes the paper by summarizing the key points and providing some perspectives for future research.

\section{Modelling and main results} \label{section2}~\par
Let $ T > 0 $ denote a fixed horizon of time.
\subsection{Mathematical model}~~\par
The model for infected and non-infected mosquitoes proposed in \cite{strugarek2016reduction} can be used to derive problem \eqref{eq-0} from a unique two-populations model posed on $\mathbb{R}^d$, with $d\ge 1$. In this model, $n_i$ represents the density of infected mosquitoes, and $n_u$ represents the density of non-infected mosquitoes. The dynamics of these quantities is governed by the following system~:
\begin{EqSystem}\label{2.1}
\partial_{t}n_{i}-\Delta n_{i} = (1-s_{f})F_{u}n_{i}(1-\frac{N}{K})-\delta d_{u}n_{i},& in ~~ (0,T)\times  \mathbb{R}^d, \\\label{2.2}
\partial_{t}n_{u}-\Delta n_{u} = F_{u}n_{u}(1-s_{h}p)(1-\frac{N}{K})- d_{u}n_{u},& in ~~ (0,T)\times  \mathbb{R}^d, \\ 
N = n_{i} + n_{u}, \qquad  p=\frac{n_{i}}{n_{i}+n_{u}}.
\end{EqSystem}
This system is complemented by initial conditions $n_i(t=0,x) = n_i^{init}(x) \ge 0$,  $n_u(t=0,x) = n_u^{init}(x) > 0$. The parameters are defined as follows: 
\begin{itemize}
    \item[$\bullet$]  $d_u$, $d_i = \delta d_u$ with $\delta > 1$ represent the death rates for uninfected and infected mosquitoes, respectively. We assume that $d_i > d_u$ because Wolbachia shortens their lifespan;
    \item[$\bullet$] $s_f \in (0,1)$ is a dimensionless parameter that reflects the reduced fecundity of infected mosquitoes;
    \item[$\bullet$] $F_u$, $F_i = (1 - s_f) F_u$ denote the fecundity rates for uninfected and infected mosquitoes, respectively. We assume that $F_i < F_u$, as Wolbachia reduces fecundity;
    \item[$\bullet$] $ s_h \in (0, 1) $ is the cytoplasmic incompatibility parameter, representing the fraction of eggs from uninfected females fertilized by infected males that do not hatch. In other words, a proportion $ 1 - s_h $ of the eggs fertilized by infected males actually hatch. Cytoplasmic incompatibility is perfect when $ s_h = 1 $;
    \item[$\bullet$] $K$ denotes the environmental capacity.

\end{itemize}

All the constants listed above are assumed to be positive.
More precisely, we assume
\begin{equation}\label{hyp_coeff}
\delta>1, \quad 0<s_f<1, \quad 0<s_h \leq 1.
\end{equation}
We refer to e.g. \cite{almeida2019optimal,Vasilieva} and references therein for the analysis of similar dynamical systems without spatial diffusion.
In particular, we have that under the assumption
\begin{equation}\label{s_h}
s_f + \delta - 1 < s_h \delta,
\end{equation}
there are 4 equilibria, the extinction and the co-existence equilibria are unstable, whereas the \textit{Wolbachia}-invasion equilibria and the equilibria without \textit{Wolbachia} are stable.
Existence and uniqueness of solutions for such a reaction–diffusion system are by now well-known, see e.g. \cite{evans2010partial, perthame2015parabolic}.
It is worth noting that in the reaction term of the second equation, the term $ -\frac{n_{i}}{n_{i}+n_{u}}$ represents the vertical transmission of the disease, while the coefficient $ s_h $ models the imperfection of this transmission due to cytoplasmic incompatibility.


\subsection{Formal reduction for large asymptotics}
We assume high fecundity, as described in \cite{strugarek2016reduction}, and introduce a small parameter $ \epsilon \ll 1 $ such that $ F_u = \frac{\tilde{F_u}}{\epsilon}.$ When fecundity becomes dominant compared to the other parameters, it is relevant to consider the asymptotic regime $ \epsilon \to 0 $ in order to simplify the system equations \eqref{2.1} and \eqref{2.2}.

This reduction method is inspired by \cite{strugarek2016reduction,almeida2019optimal}, and we extend their results in the case of non constant parameters.
In this section, we elaborate on the reduction procedure of this system in detail.
Since $n_i$ and $n_u$  will depend on $\epsilon$, we use the notation $n^{\epsilon}_i$ and $n^{\epsilon}_u$. Before we begin rigorous reasoning, we use the following assumptions:
\begin{hyp}\label{hypK}
  \begin{itemize}
  \item[1-] $K$ in $C^2(\mathbb{R}^d)\cap L^{\infty}(\mathbb{R}^d)$.
  \item[2-] $K$ is bounded from below, i.e., there exists $c' > 0$ such that $K > c'$, for all $x\in \mathbb{R}^d$.
  \item[3-] $\nabla K \in L^{2}(\mathbb{R}^d)$ and $\Delta K \in L^{1}(\mathbb{R}^d)$.
  \end{itemize}   
\end{hyp}

\textbf{Formal reasoning.} We introduce the total population as $N^{\epsilon} = n^{\epsilon}_i + n^{\epsilon}_u$ and the fraction of infected mosquitoes as $p^{\epsilon} = \frac{n^{\epsilon}_i}{n^{\epsilon}_i + n^{\epsilon}_u}$. After straightforward computations from \eqref{2.1} and \eqref{2.2}, we obtain the system  
\begin{align}
  \label{equationN}
  \partial_{t}N^{\epsilon}-& \Delta N^{\epsilon} =  \\
 &  N^{\epsilon} \left( F_{u}(1-\frac{N^{\epsilon}}{K})((1-s_{f})p^{\epsilon}+(1-p^{\epsilon})(1-s_hp^{\epsilon}))-d_{u}(\delta p^{\epsilon}+1-p^{\epsilon})\right)  \nonumber  \\
  \label{equationp}
\partial_{t}p^{\epsilon}-& \Delta p^{\epsilon} -2\frac{\nabla p^{\epsilon}\nabla N^{\epsilon}}{N^{\epsilon}} = p^{\epsilon}(1-p^{\epsilon})\left(F_{u}(1-\frac{N^{\epsilon}}{K})(s_{h}p^{\epsilon}-s_{f})+d_{u}(1-\delta)\right).
\end{align}
Recalling the notation $F_u = \frac{\tilde{F_u}}{\epsilon}$, we reformulate equation \eqref{equationN} as
\begin{equation}\label{eq_large fonct}
    \partial_{t}N^{\epsilon}-\Delta N^{\epsilon} = N^{\epsilon}\left[\frac{\tilde{F_{u}}}{\epsilon}(1-\frac{N^{\epsilon}}{K})((1-s_{f})p^{\epsilon}+(1-p^{\epsilon})(1-s_hp^{\epsilon}))-d_{u}(\delta p^{\epsilon}+1-p^{\epsilon})\right].
\end{equation}
We expect that, as $\epsilon \to 0$, we have $N^{\epsilon}(t,x) = K(x)(1+O(\epsilon))$. Then we introduce an expansion 
\begin{equation*}
    N^{\epsilon}(t,x) =K(x)(1-\epsilon z^{\epsilon}(t,x)).
\end{equation*}
Equating the leading terms in \eqref{eq_large fonct} yields 
\begin{equation*}
    z^{\epsilon}(t,x)=\frac{d_{u}[(\delta-1)p^{\epsilon}(t,x)+1]-\frac{\Delta K(x)}{K(x)}}{\tilde{F_{u}}((1-s_{f})p^{\epsilon}(t,x)+(1-p^{\epsilon}(t,x))(1-s_{h}p^{\epsilon}(t,x)))}.
\end{equation*}
Equation \eqref{equationp} then rewrites 
\begin{align*}\label{eq:main}
&\partial_{t}p^{\epsilon}-\Delta p^{\epsilon} -2\frac{\nabla K}{K}\nabla p^{\epsilon}-2\nabla \ln(1-\epsilon z^{\epsilon})\\&  =  p^{\epsilon}(1-p^{\epsilon})\left((s_{h}p^{\epsilon}-s_{f})\frac{d_{u}[(\delta-1)p^{\epsilon}+1]-\frac{\Delta K}{K}}{(1-s_{f})p^{\epsilon}+(1-p^{\epsilon})(1-s_{h}p^{\epsilon})}-d_{u}(\delta-1)\right),
\end{align*}
which is equivalent to

\begin{equation}\label{equation3}
\partial_{t}p^{\epsilon} - \Delta p^{\epsilon} - 2\frac{\nabla K}{K}\nabla p^{\epsilon}-2\nabla p^{\epsilon} \nabla \ln(1-\epsilon z^{\epsilon}) = f(p^{\epsilon}) + \frac{\Delta K}{K} g(p^{\epsilon}),
\end{equation}

where
 $$f(p^{\epsilon}) = p^{\epsilon}(1-p^{\epsilon})\left[\frac{d_u(\delta s_h p^{\epsilon} -s_f+1-\delta)}{(1-s_{f})p^{\epsilon}+(1-p^{\epsilon})(1-s_{h}p^{\epsilon})}\right],$$
 and $$g(p^{\epsilon}) = \frac{-p^{\epsilon}(1-p^{\epsilon})(s_{h}p^{\epsilon}-s_{f})}{(1-s_{f})p^{\epsilon}+(1-p^{\epsilon})(1-s_{h}p^{\epsilon})}.$$
Before stating the convergence theorem, we will make a remark on the property of bistability:
\begin{remkm}\label{rem:bistable}
  The function $f $ is continuous and Lipschitz continuous, with $f(0) = f(1) = f(\theta_1) = 0 $. Furthermore, $f < 0$  on the interval  $(0, \theta_1)$  and $f > 0$  on  $(\theta_1, 1)$, where  $\theta_1 = \frac{{(\delta-1) + s_f}}{{s_h\delta}} \in (0, 1)$, under condition \eqref{s_h}.
  We say that $f$ is \textit{bistable}.
\end{remkm}
 
\begin{hyp}\label{ass:init}
We assume that the initial conditions satisfy:
\begin{align}
& n^{init,\epsilon}_{i} \text{ is uniformly bounded in } L^{\infty}(\mathbb{R}^d), \quad n^{init,\epsilon}_{i} \ge 0, \label{ini_1}\\
&  n^{init,\epsilon}_{u} \text{ is uniformly bounded in } L^{\infty}(\mathbb{R}^d), \quad n^{init,\epsilon}_{u} > 0, \\
  &  p^{init,\epsilon} = \frac{n^{init,\epsilon}_{i}}{n^{init,\epsilon}_{i}+ n^{init,\epsilon}_{u}} \text{ is uniformly bounded in }  L^2(\mathbb{R}^d), \\
  & \text{ and }
p^{init,\epsilon}  \xrightarrow{\epsilon \to 0} p^{init,0}
   \quad \text{weakly in } L^2(\mathbb{R}^d). \notag\label{ini_p}
\end{align}
Additionally, we assume that
\begin{equation}\label{ini_2}
n^{init,\epsilon}_{i}+ n^{init,\epsilon}_{u} = K+\epsilon K^{\epsilon}_{0}, 
\quad \| K^{\epsilon}_{0}\|_{\infty} \le C'.
\end{equation}
\end{hyp}

We are now in position to present our first main result, namely the reduction of system \eqref{equationN}-\eqref{equationp} to the scalar equation \eqref{equation3} as $\epsilon \to 0$. This result extends the one presented in \cite{strugarek2016reduction}, where the carrying capacity is assumed to be constant.

\begin{thmm} \label{thm-conv}
  Let $ T > 0$ and assume \eqref{hyp_coeff}--\eqref{s_h} hold. Under Assumptions \ref{hypK} and \ref{ass:init}, as $\epsilon \rightarrow 0$, we have $ p^{\epsilon} := \frac{n^{\epsilon}_i}{n^{\epsilon}_i + n^{\epsilon}_u} \to p^0$ strongly in $L^2(0,T,L^2(\mathbb{R}^d))$, weakly in $ L^2(0,T,H^{1}(\mathbb{R}^d))$ where $p^0$ is the unique solution of 
  \begin{System}\label{eq_etude}
    \partial_{t}p^0 - \Delta p^0 - 2\frac{\nabla K}{K}\nabla p^0 = f(p^0) + \frac{\Delta K}{K} g(p^0),& in ~~ (0,T)\times \mathbb{R}^d,\\[2mm]
    p^0(t=0,x) = p^{init}(x),&in ~~ \mathbb{R}^d,
  \end{System}
  with
  \begin{equation}\label{def:f}
    f(p) = p(1-p)\left[\frac{d_u(\delta s_h p -s_f+1-\delta)}{(1-s_{f})p+(1-p)(1-s_{h}p)}\right],
  \end{equation} 
  and   
  \begin{equation}\label{def:g}
    g(p) = \frac{-p(1-p)(s_{h}p-s_{f})}{(1-s_{f})p+(1-p)(1-s_{h}p)}.
  \end{equation}
\end{thmm}

\subsection{Results on the blocking}\label{main-resul-3}
In this subsection, we present our main result, which addresses blocking the propagation of Wolbachia in heterogeneous environments where the environmental carrying capacity varies significantly across different regions of the spatial domain. Before presenting our main result for equation \eqref{eq_etude}, we introduce the concept of propagation barrier (referred to simply as barrier below).
For the sake of simplicity, we consider the domain to be one-dimensional and therefore study the differential equation : 
\begin{equation}\label{notre systeme}
  \partial_{t}p - \partial_{xx} p - 2\frac{K^{'}}{K}    \partial_{x} p = f(p) + \frac{K^{''}}{K} g(p),\quad t \ge 0, ~~ x \in\mathbb{R},
\end{equation}
where $f$ and $g$ are defined in \eqref{def:f} and \eqref{def:g} respectively.
To proceed, let us denote by $F$ the antiderivative of $f$ that vanishes at $0$, i.e.
\begin{equation*}
  F(x) = \int_{0}^{x} f(t)\, dt.
\end{equation*} 
Throughout this paper, we assume that
\begin{equation*}
  F(1) = \int_{0}^{1} f(t)\, dt > 0.
\end{equation*}
In the case where $K$ is constant, it is well known, since $f$ is bistable (see Remark \ref{rem:bistable}), that there exists a unique traveling wave. Its speed has the sign of $F(1)$ which is positive.
Depending on the initial conditions, the solution may converge to $1$, leading to an invasion phenomenon see e.g. \cite{Zlatos, NadinToledo} and references therein.

In this paper, we are interested in what happens when $K$ depends on $x$. To obtain a tractable problem, we assume that $K$ increases exponentially in a given region of the spatial domain and remains constant elsewhere. Specifically, for some constants $C > 0$, $L > 0$,
\begin{equation}\label{equation pour K}
K(x) = \begin{cases}
A,&\text{for } x < 0,\\
    Ae^{Cx},&\text{for } x \in [0, L],\\
Ae^{CL},&\text{for } x > L.
\end{cases}
\end{equation}
In this context, our goal is to prove the blocking of Wolbachia propagation where $K$ varies significantly across different regions of the spatial domain. To achieve this, we will show that equation \eqref{notre systeme} admits a stationary solution and satisfies the comparison principle.
The existence of a stationary wave for this problem boils down to the existence of a solution to
\begin{equation}\label{equationelliptique}
\begin{cases}
-p''-2Cp' - C^{2}g(p) = f(p),&on ~~]0,L[,\\
-p''  = f(p),&on ~~[0,L]^{c},\\
p{'}(0^-)-p{'}(0^+)= Cg(p(0)),\\
p{'}(L^-)-p{'}(L^+)= -Cg(p(L)),\\
p(-\infty)=1,~p(+\infty)=0,~p > 0.
\end{cases}
\end{equation}
In our study, stationary solutions of equation \eqref{notre systeme}, which exhibit specific behavior at infinity as described by the solutions of equation \eqref{equationelliptique}, act as barriers that prevent propagation.
\begin{defim} \label{def-barier}
  We name a $(C, L)$-barrier any solution to \eqref{equationelliptique}. For any bistable function $f$ we define the barrier set
  \begin{equation*}
    B(f) = \{ (C,L) \in (0,+\infty)^2,~ there~ exists~ a~ (C, L)-barrier\}.
  \end{equation*}
\end{defim}

The following Lemma states that if $C $ is small enough, there is no blockage.
\begin{lemma}\label{c small}
  If $ \frac{ K^{'}}{K} $ and $ \frac{K^{''}}{K} $ are sufficiently small, then there exists a sub-solution $p_A$ of equation \eqref{notre systeme} that is not blocked, i.e $p_A$ converges to $1$ as $t \to +\infty$ uniformly on any compact.
\end{lemma}\\
  
Next, for suitable values of $C$ and $L$, barriers can be constructed.

\begin{thmm}\label{thmprincipale}
  Assume that $K$ is given by \eqref{equation pour K}, that \eqref{hyp_coeff}--\eqref{s_h} hold, and that $F(1)>0$.
   Then there exist $C^*>0$ such that for any $C>C^*$ there exists a $(C,L)$ barrier for some $L>0$, in the sense of Definition \ref{def-barier}.
\end{thmm}\\

We will also show later that if a $(C_1,L_1)$-barrier exists, then barriers also exist for any larger length $L_2>L_1$, provided that $C_2$ is chosen according to the relation $C_2= C_1\frac{L_1}{L_2}$.

The existence of a barrier significantly and directly influences the asymptotic behavior of the solutions of equation \eqref{notre systeme}, as established by the following proposition.
\begin{propm}\label{principe de comparaison}
Suppose that there exists a barrier $(C, L)$ denoted by $\overline{P}$ and that $ K $ is given by \eqref{equation pour K}. Then every weak solution of
\begin{equation*}
\partial_{t}p - \partial_{xx} p - 2\frac{ K^{'}}{K}\partial_{x} p = f(p) + \frac{K^{''}}{K} g(p),
\end{equation*}
with initial data such that $p^{ini} \le \overline{P}$, verifies $\forall t > 0$, $p(t) \le \overline{P}$, then it stops spreading.
\end{propm}

\section{Reduction}\label{section3}

Let us first recall the compactness result (see \cite{simon1986compact}).
\begin{lemm}[Lions-Aubin] \label{AubinLions} 
  Let $T > 0$, $q \in (1,\infty)$, $(\psi_{n})_{n}$ a bounded sequence in $L^q(0,T,H)$, where $H$ is a banach space. If $\psi_{n}$ is bounded in $L^q(0,T,V)$ and $V$ compactly embeds in $H$, and if $(\partial_{t}\psi_{n})_{n}$ is bounded in $L^q(0,T,V')$ uniformly with respect to $n$, then $(\psi_{n})_{n}$ is relatively compact in $L^q(0,T,H)$.
\end{lemm}\\

\textbf{Proof of Theorem \ref{thm-conv} :} We split the proof of the theorem into several steps.\\[2mm]
\textbf{Step 1 : Uniform a priori estimates.} We initially establish  some uniform bounds with respect to $\epsilon > 0$.\\
Before starting with the first lemma which proves that $z^{\epsilon}$ is uniformly bounded with respect to $\epsilon$, we introduce the following system of equations, which is used to prove this property.
We have  $N^{\epsilon} = K(x)(1-\epsilon z^{\epsilon}(t,x))$ and  $p^{\epsilon} = \frac{n^{\epsilon}_i}{n^{\epsilon}_i + n^{\epsilon}_u}$. \\ Let us replace $N^{\epsilon}$ with its expression in equation \eqref{eq_large fonct}, and then, from straightforward computations, we get
\begin{align}\label{eq-zepsilon}
  & \partial_{t}z^{\epsilon}-\Delta z^{\epsilon}-2\frac{\nabla K}{K} \nabla  z^{\epsilon} = \\
  & \quad -\frac{(1-\epsilon z^{\epsilon})}{\epsilon}\left[\frac{\tilde{F_{u}}}{\epsilon}(1-\frac{N^\epsilon}{K})((1-s_{f})p^{\epsilon}+(1-p^{\epsilon})(1-s_hp^{\epsilon}))-d_{u}(\delta p^{\epsilon}+1-p^{\epsilon}) \right]  \notag \\
  & \quad\quad -\frac{(1-\epsilon z^{\epsilon})}{\epsilon}\frac{\Delta K}{K}\notag.\end{align}

\begin{lemm}
Assume that the assumptions of Theorem \ref{thm-conv} hold, and let $(z^{\epsilon},p^{\epsilon})$ be the unique solution of system \eqref{eq-zepsilon}-\eqref{equationp}. Then 
    $z^{\epsilon}$  is uniformly bounded with respect to $\epsilon$, and $0 \le p^{\epsilon} \le 1$ on  $[0,T]\times\mathbb{R}^d$.
\end{lemm}
\begin{proof}
  By nonnegativity of $n^{\text{init},\epsilon}_i$ and $n^{\text{init},\epsilon}_u$, it is standard to deduce the nonnegativity of $n^{\epsilon}_i$ and $n^{\epsilon}_u$ (Indeed $0$ is a subsolution for \eqref{2.1}--\eqref{2.2}). 
  Moreover, by the strong maximum principle we have $n^{\epsilon}_u > 0$ on $(0,T) \times\mathbb{R}^d $. 
Therefore, $p^{\epsilon}$ is well-defined on $[0, T] \times \mathbb{R}^d$ and satisfies by definition $0 \leq p^{\epsilon} \leq 1$ on $[0, T] \times \mathbb{R}^d$.

Consider the function $h$ defined as  $h(p) = \frac{d_{u}(\delta p+1-p)}{\tilde{F_{u}}((1-s_{f})p+(1-p)(1-s_hp))}$. We remark that the denominator is positive, since $p \in [0,1]$.
Let
\begin{equation*}
  \tilde{K} = \max\{ \max_{p \in [0,1]} (h(p)-\frac{\min\frac{\Delta K}{K}}{\tilde{F_{u}}((1-s_{f})p+(1-p)(1-s_hp))}), \| z^{init,\epsilon}\|_{\infty}\}.
\end{equation*}
Let $\epsilon_0$ be such that $ \tilde{K} \le \frac{1}{\epsilon_0}$( $\Leftrightarrow  1-\epsilon_0 \tilde{K} \ge 0$), then, thanks to this choice  we have, for $0 < \epsilon \le \epsilon_0 $,
\begin{align*}
  & -\frac{(1-\epsilon \tilde{K})}{\epsilon}\left(\tilde{F_{u}}\tilde{K}((1-s_{f})p^{\epsilon}+(1-p^{\epsilon})(1-s_hp^{\epsilon}))-d_{u}(\delta p^{\epsilon}+1-p^{\epsilon}) + \frac{\Delta K}{K}\right)  \\
  & = -\frac{(1-\epsilon \tilde{K})}{\epsilon}\left(\tilde{F_{u}}[((1-s_{f})p^{\epsilon}+(1-p^{\epsilon})(1-s_hp^{\epsilon}))][\tilde{K}-h(p^{\epsilon})]+\frac{\Delta K}{K}\right).
\end{align*}
According to the definition of $ \tilde{K}$, we have
\begin{align*}
  & -\frac{(1-\epsilon \tilde{K})}{\epsilon}(\tilde{F_{u}}[((1-s_{f})p^{\epsilon}+(1-p^{\epsilon})(1-s_hp^{\epsilon}))][\tilde{K}-h(p^{\epsilon})]+\frac{\Delta K}{K})  \\
  & \qquad \le 0 = \partial_{t}\tilde{K}-\Delta \tilde{K}-2\frac{\nabla K}{K} \nabla \tilde{K}.
\end{align*} 
Hence, $\tilde{K}$ is a super-solution for \eqref{eq-zepsilon} for any $0 < \epsilon \le \epsilon_0 $. Then  $z^{\epsilon} \le \tilde{K}$ for any $0 < \epsilon \le \epsilon_0 $.\\ By the same token, we have that the negative constant 
\begin{equation*}
\underset{\sim}{K} = \min\{ \min_{p \in [0,1]} (h(p) - \frac{\max\frac{\Delta K}{K}}{\tilde{F}_{u}\big((1-s_{f})p+(1-p)(1-s_h p))}), -\|z^{\text{init},\epsilon}\|_{\infty}\} \end{equation*}
is a sub-solution for \eqref{eq-zepsilon}. Indeed, Let $\epsilon_0$ be such that $ \underset{\sim}{K} \le \frac{1}{\epsilon_0}$ ($\Leftrightarrow  1-\epsilon_0 \underset{\sim}{K} \ge 0$), then, thanks to this choice  we have for  $0 < \epsilon \le \epsilon_0$:
\begin{align*}
& -\frac{(1-\epsilon \underset{\sim}{K})}{\epsilon}\left[
\tilde{F}_{u} \underset{\sim}{K} 
\big((1-s_{f})p^{\epsilon} + (1-p^{\epsilon})(1-s_h p^{\epsilon})\big) 
- d_{u}(\delta p^{\epsilon} + 1 - p^{\epsilon}) + \frac{\Delta K}{K}\right]  \\
& = -\frac{(1-\epsilon \underset{\sim}{K})}{\epsilon}
\bigg(\tilde{F}_{u} \big[(1-s_{f})p^{\epsilon} + (1-p^{\epsilon})(1-s_h p^{\epsilon})\big] 
\big[\underset{\sim}{K} - h(p^{\epsilon})\big] 
+ \frac{\Delta K}{K}\bigg).
\end{align*}
According to the definition of $\underset{\sim}{K}$, we have:
\begin{align*}
& -\frac{(1-\epsilon \underset{\sim}{K})}{\epsilon}
\bigg(\tilde{F}_{u} \big[(1-s_{f})p^{\epsilon} + (1-p^{\epsilon})(1-s_h p^{\epsilon})\big]
\big[\underset{\sim}{K} - h(p^{\epsilon})\big]
                 + \frac{\Delta K}{K}\bigg) \\
& \qquad \geq 0 = \partial_{t}\underset{\sim}{K} - \Delta \underset{\sim}{K} 
- 2\frac{\nabla K}{K} \nabla\underset{\sim}{K}.
\end{align*}

Hence, we conclude that $\underset{\sim}{K}$ is a sub-solution for \eqref{eq-zepsilon}. Thus, $z^{\epsilon}$ is uniformly bounded from below by $\underset{\sim}{K}$. We therefore deduce the uniform bound of $z^{\epsilon}$.
 \end{proof}
\begin{lemm}\label{lemma estimate}
  We assume that Assumptions \ref{hypK} holds. Let $(z^{\epsilon},p^{\epsilon})$ represent the unique solution of the equations \eqref{eq-zepsilon}-\eqref{equationp}. For any sufficiently small $\epsilon > 0$, the following uniform estimate holds~:
   \begin{equation}\label{estimation-z}
      \epsilon \int_{0}^{T}\int_{\mathbb{R}^d}|\nabla z^{\epsilon} |^2dxdt \le S_{1},
   \end{equation}
   and
   \begin{equation}\label{estimation-grad p}
       \int_{0}^{T}\int_{\mathbb{R}^d}|\nabla p^{\epsilon}|^2dxdt \le S_{2},
     \end{equation}
     for some positive constant $S_{k}$, $k = 1,2$.
\end{lemm}
\begin{proof}
We begin by proving \eqref{estimation-z}.
 We multiply equation \eqref{eq-zepsilon} by $ z^{\epsilon}$ and integrate it over $\mathbb{R}^{d}$, to get
\begin{align*}
 &\frac{\epsilon}{2}\frac{d}{dt}\int_{\mathbb{R}^d} (z^{\epsilon})^2dx +\epsilon\int_{\mathbb{R}^d}|\nabla z^{\epsilon}|^2dx \\&\le  \int_{\mathbb{R}^d} \Big|\big[ z^{\epsilon}(1-\epsilon z^{\epsilon}(t,x))[\tilde{F_{u}}z^{\epsilon}((1-s_{f})p^{\epsilon}+(1-p^{\epsilon})(1-s_hp^{\epsilon}))- d_{u}(\delta p^{\epsilon}+1-p^{\epsilon})]\\&\quad\quad\quad+ z^{\epsilon}(1-\epsilon z^{\epsilon})\frac{\Delta K}{K}\big]\Big|\,dx
    + 2\epsilon \int_{\mathbb{R}^d}|z^{\epsilon}| \left|\frac{\nabla K}{K}\right| |\nabla  z^{\epsilon}|\,dx.
\end{align*}
  We have proved that $ z^{\epsilon}$ and $p^{\epsilon}$ are uniformly bounded, and under Assumptions \ref{hypK}, there exist two positive constants $S_3$ and $S_4$ such that  
  \begin{equation*}
    \frac{\epsilon}{2}\frac{d}{dt}\int_{\mathbb{R}^d} (z^{\epsilon})^2dx +\epsilon\int_{\mathbb{R}^d}|\nabla z^{\epsilon}|^2dx \le S_3 +2S_4\epsilon\int_{\mathbb{R}^d} |\nabla K||\nabla  z^{\epsilon}|dx. 
  \end{equation*}
Using a Cauchy-Schwarz  inequality and a Young inequality, we deduce
\begin{equation*}
\frac{\epsilon}{2}\frac{d}{dt}\int_{\mathbb{R}^d} (z^{\epsilon})^2dx +(\frac{3\epsilon}{4})\int_{\mathbb{R}^d}|\nabla z^{\epsilon}|^2dx  \le S_5.  
\end{equation*}
We obtain \eqref{estimation-z} through time integration.\\

Before proving estimation \eqref{estimation-grad p}, we show that 
$ \int_{0}^{T}\int_{\mathbb{R}^d}|\nabla \ln(1-\epsilon z^{\epsilon})|^2 dx dt $ approaches 0 as $\epsilon$ tends to 0. We have 
\begin{equation*}
  \int_{0}^{T}\int_{\mathbb{R}^d}|\nabla \ln(1-\epsilon z^{\epsilon})|^2 dx dt =   \epsilon^2\int_{0}^{T}\int_{\mathbb{R}^d} \frac{ |\nabla  z^{\epsilon}|^2}{|1-\epsilon z^{\epsilon}|^2}dxdt. 
\end{equation*}
Since  $z^{\epsilon} $ is  bounded, there exists a constant $S_6$ such that,
\begin{equation*}
  \int_{0}^{T}\int_{\mathbb{R}^d}|\nabla \ln(1-\epsilon z^{\epsilon})|^2 \,dxdt \le  \epsilon^2 S_6\int_{0}^{T}\int_{\mathbb{R}^d}  |\nabla  z^{\epsilon}|^2\,dxdt.  
\end{equation*}
Then, according to estimate \eqref{estimation-z}, the right hand side converges to $0$ as $\epsilon $ goes to $0$.

To show estimate \eqref{estimation-grad p}, we multiply equation \eqref{equation3} by $p^{\epsilon}$ and integrate it over $\mathbb{R}^d$ to get
\begin{align*}
  \frac{1}{2}\frac{d}{dt}\int_{\mathbb{R}^d} (p^{\epsilon})^2\,dx
  & +\int_{\mathbb{R}^d}|\nabla p^{\epsilon}|^2\,dx - 2\int_{\mathbb{R}^d}p^{\epsilon}\frac{\nabla K}{K}\nabla p^{\epsilon}\,dx 
  -2\int_{\mathbb{R}^d}p^{\epsilon}\nabla p^{\epsilon} \nabla \ln(1-\epsilon z^{\epsilon})\, dx \\
  & \le \int_{\mathbb{R}^d} p^\epsilon f(p^\epsilon)\,dx  + \frac{\max_{p \in [0,1]}|g(p)|}{c'}\int_{\mathbb{R}^d}|\Delta K|\,dx.
\end{align*}
Using a Cauchy-Schwarz inequality, we get 
\begin{align*}
  &\frac{1}{2}\frac{d}{dt}\int_{\mathbb{R}^d} (p^{\epsilon})^2dx +\int_{\mathbb{R}^d}|\nabla p^{\epsilon}|^2dx \\
  &\le \frac{1}{c'}\left(\int_{\mathbb{R}^d}|\nabla p^{\epsilon}|^2\,dx\right)^{\frac{1}{2}}
    \left(\int_{\mathbb{R}^d}|2\nabla K|^2\,dx \right)^{\frac{1}{2}}
    + \int_{\mathbb{R}^d} p^\epsilon f(p^\epsilon)\,dx \\
  & \quad + \frac{2}{3}\left(\int_{\mathbb{R}^d}|\nabla p^{\epsilon}|^2\,dx\right)^{\frac{1}{2}}
    \left(\int_{\mathbb{R}^d}|3\nabla \ln(1-\epsilon z^{\epsilon})|^2dx\right)^{\frac{1}{2}}\\
  &\quad + \frac{\max_{p \in [0,1]}|g(p)|}{c'}\int_{\mathbb{R}^d}|\Delta K|\,dx.
\end{align*}
From assumptions \ref{hypK} and the well-known inequality $2ab\le a^2+b^2$, we deduce after an integration in time
\begin{align}
  \label{ineq1}
  & \frac{1}{2}\int_{\mathbb{R}^d} (p^{\epsilon})^2\,dx +\frac{1}{6}\int_{0}^{T}\int_{\mathbb{R}^d}|\nabla p^{\epsilon}|^2\,dxdt \\
  &\le S' + \int_{\mathbb{R}^d} p^\epsilon f(p^\epsilon)\,dx +\int_{0}^{T}\int_{\mathbb{R}^d}|\nabla \ln(1-\epsilon z^{\epsilon})|^2\,dxdt. \nonumber
\end{align}
We note that $p^\epsilon f(p^\epsilon) \le C_1 (p^\epsilon)^2$ for some constant $C_1>0$. Therefore,
\begin{equation*}
\int_{\mathbb{R}^d} p^\epsilon f(p^\epsilon) \, dx \le C_1 \int_{\mathbb{R}^d} (p^\epsilon)^2 \, dx.
\end{equation*}
Inequality \eqref{ineq1} can thus be rewritten as
\begin{equation*}
  \frac{1}{2}\int_{\mathbb{R}^d} (p^{\epsilon})^2dx +\frac{1}{6}\int_{0}^{T}\int_{\mathbb{R}^d}|\nabla p^{\epsilon}|^2dxdt \le S' +  C_1 \int_{\mathbb{R}^d} (p^\epsilon)^2 \, dx +\int_{0}^{T}\int_{\mathbb{R}^d}|\nabla ln(1-\epsilon z^{\epsilon})|^2\,dxdt.
\end{equation*}
We then apply Grönwall's lemma to the inequality
\begin{equation*}
\frac{1}{2} \frac{d}{dt} \int_{\mathbb{R}^d} (p^\epsilon)^2 \, dx \le C \int_{\mathbb{R}^d} (p^\epsilon)^2 \, dx.
\end{equation*}
It provides a uniform bound for 
$\int_{\mathbb{R}^d} (p^\epsilon)^2 \, dx$ in terms of the initial data. Finally, by taking $\epsilon$ small enough, we obtain the desired estimate.\\

\textbf{Step 2 : Local relative Compactness.} For $R > 0$, we define the increasing sequence $(B_R)_{R}$ of balls of radius $R$ with center $0$ in $\mathbb{R}^d$. We use Lions-Aubin Lemma \ref{AubinLions} with $q=2$, $H_{R} = L^2(B_R)$ and $V_{R}= H^1(B_R)\cap L^{\infty}(B_R)$.
The sequence $(p^{\epsilon})$ is uniformly bounded in   $L^2(0,T,V_{R})$  from Lemma \ref{lemma estimate}. The compact embedding of $V_{R}$ in $H_{R}$ is well-known  from the Rellich-Kondrachov theorem. We are left to verify the bound on the time derivative : Let $\phi \in  V_{R}$, we denote $\langle\cdot,\cdot\rangle = \langle\cdot,\cdot\rangle_{V_{R}',V_{R}}$, the duality bracket. From equation \eqref{equation3}, we get
\begin{equation*}
  \int_{0}^{T}|\langle\partial_{t}p^{\epsilon}(t),\phi\rangle|^2\,dt
  = \int_{0}^{T}\left|\langle \Delta p^{\epsilon} + 2\frac{\nabla K}{K}\nabla p^{\epsilon} + 2\nabla p^{\epsilon} \nabla \ln(1-\epsilon z^{\epsilon})-\psi^{\epsilon},\phi\rangle\right|^2\,dt,
\end{equation*}
where $\psi^{\epsilon}$ is the function defining the right hand side in equation \eqref{equation3}, which is uniformly bounded in $L^1((0,T)\times B_{R})\cap L^{\infty}((0,T)\times B_{R})$.
Then,
\begin{align*}
  & \int_{0}^{T}\left|\langle\Delta p^{\epsilon} + 2\frac{\nabla K}{K}\nabla p^{\epsilon} + 2\nabla p^{\epsilon} \nabla \ln(1-\epsilon z^{\epsilon})-\psi^{\epsilon},\phi\rangle\right|^2\,dt \\
  & < \|\nabla\phi\|^2_{L^2(B_{R})}\int_{0}^{T}\|\nabla p^{\epsilon}\|^2_{L^2(B_{R})} \, dt 
    +\frac{2}{c'} \|\phi\|^2_{L^{\infty}(B_{R})} \|\nabla K\|^2_{L^2(B_{R})}
    \int_{0}^{T} \|\nabla p^{\epsilon}\|^2_{L^2(B_{R})}\,dt  \\
  &\quad  +2\|\phi\|^2_{L^{\infty}(B_{R})}
    \int_{0}^{T}\|\nabla \ln(1-\epsilon z^{\epsilon})\|^2_{L^2(B_{R})} \|\nabla p^{\epsilon}\|^2_{L^2(B_{R})}\,dt + \|\phi\|^2_{L^2(B_{R})}\|\psi^{\epsilon}\|^2_{L^{\infty}(B_{R})}.
\end{align*}

\textbf{Step 3 : Global convergence.} For all $R \in \mathbb{N}$, one can extract from $(p^\epsilon)_\epsilon$ a subsequence converging in $L^{2}(0,T,H_{R})$, by Lions-Aubin lemma. We perform a diagonal extraction process successively in $R$, so that
\begin{equation*}
  p^{\epsilon_{m}^{(R)}}\xrightarrow{m \to +\infty} p^{(R)}\quad in~~L^{2}(0,T,H_{R}),
\end{equation*}
and by construction $(\epsilon_{m}^{(R_1)})_{m}$ is a subsequence of $(\epsilon_{m}^{(R_2)})_{m}$ if $R_2 > R_1$. This ensures the consistency of the subsequences on increasing nested domains. Since the family $(p^{\epsilon})_{\epsilon}$ is uniformly bounded in $L^{2}(0,T;H^1(\mathbb{R}^d))$, we can also obtain the weak convergence of the gradient:  
\begin{equation*}
    \nabla p^{\epsilon_{m}^{(R)}} \xrightarrow{m \to +\infty}\nabla p^{(R)} \quad \text{in } L^{2}(0,T;H_{R}).
\end{equation*}
The limits $p^{(R)}$ are well defined, independent of the extracted subsequences, and satisfy the same bounds as $(p^{\epsilon})_{\epsilon}$. Additionally,  
\begin{equation*}
    R_2 > R_1 \implies p^{(R_2)}\big|_{B_{R_1}} = p^{(R_1)}.
\end{equation*}
This property allows us to define a global limit function $p^{0} \in L^{2}(0,T;L^{2}(\mathbb{R}^d))$ by combining the local limits $p^{(R)}$ for $R > 0$. Finally, we have constructed a subsequence, still denoted as $(p^{\epsilon})$, such that  
\begin{equation*}
    p^{\epsilon} \xrightarrow{\epsilon \to 0} p^{0} \quad \text{strongly in } L^{2}(0,T;L^{2}(B_R)), \quad \forall R > 0.
\end{equation*}
To transition from local convergence to global convergence, it is necessary to obtain uniform estimates in $\epsilon$ for the tails, i.e., for $|x| > R$. To achieve this, we introduce a function $\phi \in C^{\infty}(\mathbb{R}^{d})$ such that $0 \leq \phi \leq 1$, $\phi(x) = 0$ if $|x| < \frac{1}{2}$, and $\phi(x) = 1$ if $|x| > 1$. We then define $\phi_{R}(x) = \phi\left(\frac{x}{R}\right)$.
Multiplying \eqref{equation3} by $p^{\epsilon} \phi_{R}$ and integrating over $\mathbb{R}^d$, we obtain
\begin{align*}
  \frac{1}{2} \frac{d}{dt} \int_{\mathbb{R}^d}(p^{\epsilon})^2\phi_{R}\, dx
  &\le -\int_{\mathbb{R}^d}\nabla(p^{\epsilon}\phi_{R})\nabla p^{\epsilon} \, dx + 2\int_{\mathbb{R}^d}\frac{\nabla K}{K}\nabla p^{\epsilon} p^{\epsilon} \phi_{R}\, dx \\
  &\quad+ 2 \int_{\mathbb{R}^d}\nabla p^{\epsilon}\nabla \ln(1-\epsilon z^{\epsilon}) p^{\epsilon} \phi_{R}\, dx + A \int_{\mathbb{R}^d} (p^{\epsilon})^2 \phi_{R}\, dx,
\end{align*}
where
\begin{align*}
  A = \sup_{\substack{p \in [0,1]}}
  & \left( \left|\frac{(1-p)(s_{h}p - s_{f}) d_{u}[(\delta-1)p + 1]}{(1-s_{f})p + (1-p)(1-s_{h}p)} - d_{u}(\delta-1)\right| \right.  \\
  & \quad + \left. \left|\frac{\max\left[\frac{\Delta K}{K}\right] (1-p)(s_{h}p - s_{f})}{(1-s_{f})p + (1-p)(1-s_{h}p)}\right| \right).
\end{align*}
Expanding the first term of the right hand side of this latter inequality, we get
\begin{equation*}
\int_{\mathbb{R}^d} \nabla(p^{\epsilon} \phi_R) \cdot \nabla p^{\epsilon} \, dx = \int_{\mathbb{R}^d} p^{\epsilon} \nabla \phi_R \cdot \nabla p^{\epsilon} \, dx + \int_{\mathbb{R}^d} \phi_R |\nabla p^{\epsilon}|^2 \, dx.
\end{equation*}
By bounding $|\nabla \phi_R|$ by its $ L^{\infty} $ norm and applying Hölder's inequality, we get the following.
\begin{equation*}
  \left|\int_{\mathbb{R}^d} p^{\epsilon} \nabla \phi_R \cdot \nabla p^{\epsilon} \, dx\right|
  \leq \|\nabla \phi\|_{L^{\infty}} \int_{\mathbb{R}^d} |p^{\epsilon}  \cdot \nabla p^{\epsilon} |\, dx.
\end{equation*}
Then, using Young's inequality, we deduce
\begin{equation*}
\left\|\int_{\mathbb{R}^d} p^{\epsilon} \nabla \phi_R \cdot \nabla p^{\epsilon} \, dx\right\| \leq \frac{\|\nabla \phi\|_{L^{\infty}}}{R}\|p^{\epsilon}\|_{L^{2}}\|\nabla p^{\epsilon}\|_{L^{2}}.
\end{equation*}







Following the same procedure and given that $ z^{\epsilon} $ is uniformly bounded, we get
\begin{align*}
\frac{1}{2} \frac{d}{dt} \int_{\mathbb{R}^d} (p^{\epsilon})^2 \phi_R \, dx 
&\leq \frac{\|\nabla \phi\|_{L^{\infty}}}{R}\|p^{\epsilon}\|_{L^{2}}\|\nabla p^{\epsilon}\|_{L^{2}} + 2 \max \left| \frac{\nabla K}{K} \right| \int_{\mathbb{R}^d} |\nabla p^{\epsilon} p^{\epsilon} \phi_R| \, dx 
\\&\quad+ \epsilon \int_{\mathbb{R}^d} |\nabla p^{\epsilon}|^2 \, dx 
+ \epsilon^2 S' \int_{\mathbb{R}^d} |\nabla z^{\epsilon}|^2 \, dx
+ A \int_{\mathbb{R}^d} (p^{\epsilon})^2 \phi_R \, dx.
\end{align*}
It implies
\begin{align*}
\frac{1}{2} \frac{d}{dt} \int_{\mathbb{R}^d} (p^{\epsilon})^2 \phi_R \, dx 
&\leq \left(\frac{1}{R}+\epsilon\right) \left(||\nabla \phi||_{L^{\infty}}||p^{\epsilon}||_{L^{2}}||\nabla p^{\epsilon}||_{L^{2}} + \epsilon S' \int_{\mathbb{R}^d} |\nabla z^{\epsilon}|^2 \, dx \right)
\\&\quad+ 2 \max \left| \frac{\nabla K}{K} \right| \int_{\mathbb{R}^d} |\nabla p^{\epsilon} p^{\epsilon} \phi_R| \, dx
+ A \int_{\mathbb{R}^d} (p^{\epsilon})^2 \phi_R \, dx.
\end{align*}
Let us define
\begin{align*}
h_{R,\epsilon}(t) &= \left(\frac{1}{R}+\epsilon\right) \left(\|\nabla \phi\|_{L^{\infty}}\|p^{\epsilon}\|_{L^{2}}\|\nabla p^{\epsilon}\|_{L^{2}} + \epsilon S' \int_{\mathbb{R}^d} |\nabla z^{\epsilon}|^2 \, dx \right)
\\&\quad\quad+ 2 \max \left| \frac{\nabla K}{K} \right| \int_{\mathbb{R}^d} |\nabla p^{\epsilon} p^{\epsilon} \phi_R| \, dx.
\end{align*}
Then, we have
\begin{equation*}
\frac{1}{2} \frac{d}{dt} \int_{\mathbb{R}^d} (p^{\epsilon})^2 \phi_R \, dx \leq h_{R,\epsilon}(t) + A \int_{\mathbb{R}^d} (p^{\epsilon})^2 \phi_R \, dx.
\end{equation*}
Thanks to Lemma \ref{lemma estimate}, we know that $ h_{R,\epsilon}$ is uniformly bounded in $L^1(0,T) $.
Using Gronwall's Lemma, we get
\begin{equation*}
\int_{\mathbb{R}^d} (p^{\epsilon})^2 \phi_R \, dx \leq e^{2At} \int_{\mathbb{R}^d} (p^{\text{init}})^2 \phi_R \, dx + \int_0^t h_{R,\epsilon}(\tau) e^{2A(t-\tau)} \, d\tau.
\end{equation*}
From the definition of $\phi_R$ (recalling that $\phi_R(x) = 1$ on $\mathbb{R}^d \setminus B_R$), we deduce that for all $\epsilon > 0$ sufficiently small and for all $R > 0$, we have
\begin{align}\nonumber
  \int_{0}^{T} \int_{\mathbb{R}^d \setminus B_R} (p^{\epsilon})^2 \, dx \, dt 
  &\leq \int_{0}^{T} \int_{\mathbb{R}^d} (p^{\epsilon})^2 \phi_R \, dx \, dt \\
  &\leq \frac{1}{2A} \left[e^{2AT} - 1\right] 
    \left[ \int_{\mathbb{R}^d \setminus B_{\frac{R}{2}}} (p^{\text{init}})^2 \, dx + \int_{0}^{T} h_{R,\epsilon}(t) \, dt \right].
    \label{eq:estimation_in_R}
\end{align}
We deduce a uniform bound, since $p^{\text{init}} \in L^2(\mathbb{R}^d)$.
Therefore, we conclude that the subsequence $(p^{\epsilon})_{\epsilon}$ converges strongly to $p^0$ in $L^2(0, T; L^2(\mathbb{R}^d))$ as $\epsilon \to 0$.
Indeed, the reasoning is as follows. We have
\begin{equation*}
    \int_{0}^T \int_{\mathbb{R}^d} |p^{\epsilon} - p^0|^2 \, dx dt = \int_{0}^T \int_{B_R} |p^{\epsilon} - p^0|^2 \, dx dt 
    + \int_{0}^T \int_{\mathbb{R}^d \setminus B_R} |p^{\epsilon} - p^0|^2 \, dx dt.
\end{equation*}
This inequality implies
\begin{align*}
  \int_{0}^T \int_{\mathbb{R}^d} |p^{\epsilon} - p^0|^2 \, dx dt  \leq
  & \int_{0}^T \int_{B_R} |p^{\epsilon} - p^0|^2 \, dx dt 
    + 2 \int_{0}^{T} \int_{\mathbb{R}^d \setminus B_R} |p^{\epsilon}|^2 \, dx dt\\ 
  & + 2 \int_{0}^{T} \int_{\mathbb{R}^d \setminus B_R} |p^0|^2 \, dx dt.
\end{align*}
The second term on the right-hand side is uniformly bounded for sufficiently large $R$ due to inequality \eqref{eq:estimation_in_R} and the fact that $p^0 \in L^2(0, T; L^2(\mathbb{R}^d))$. Regarding the first term, we rely on local convergence.\\

\textbf{Step 4 :  Passing to the limit.} Using the strong convergence of the sequence $ (p^{\epsilon}) $ in $ L^2(0, T; L^2(\mathbb{R}^d)) $ and the Lipschitz continuity of the function $I$ defined by
$$
I(p) =
p(1-p)\left[(s_{h}p-s_{f})\frac{d_{u}[(\delta-1)p+1]-D\frac{\Delta K}{K}}{(1-s_{f})p+(1-p)(1-s_{h}p)}-d_{u}(\delta-1)\right],
$$
we deduce that $ (I(p^{\epsilon})) $ converges strongly in $ L^2(0, T; L^2(\mathbb{R}^d)) $ to $ I(p^0) $,
where
Using the weak formulation of the equation satisfied by $p^{\epsilon}$ \eqref{equation3}, we derive an equation satisfied by $p^0$. Let $ \chi \in C^{\infty}_{C}((0,T)\times \mathbb{R}^d)$, we have 
\begin{align}\label{eq-limite}
  & \int_{0}^{T}\int_{\mathbb{R}^{d}} \left( p^{\epsilon}\partial_{t}\chi-\nabla p^{\epsilon} \nabla \chi-2\frac{\nabla K}{K}\nabla p^{\epsilon} \chi-2\nabla p^{\epsilon} \nabla \ln(1-\epsilon z^{\epsilon})\chi \right) \,dxdt \\
  & =  \int_{0}^{T}\int_{\mathbb{R}^{d}}I(p^{\epsilon})\chi\,dxdt.   \nonumber
\end{align}
From the above convergence it is straightforward to pass to the limit into the first two terms of the left hand side. For the fourth  term, we use the uniform bound
\begin{equation*}
  \int_{0}^{T}\int_{\mathbb{R}^{d}}|\nabla p^{\epsilon}|^2dxdt \le S',
\end{equation*}
and a Cauchy Schwarz inequality to get 
\begin{equation*}
  \int_{0}^{T}\int_{\mathbb{R}^{d}} \nabla p^{\epsilon} \nabla \ln(1-\epsilon z^{\epsilon})\chi \,dxdt
  \le S'\|\chi\|_{L^{\infty}} \| \nabla \ln(1-\epsilon z^{\epsilon})\|_{L^\infty} \xrightarrow{\epsilon \to 0} 0.
\end{equation*}
In the case where $p^{\epsilon}$ converges strongly, we can take the limit for the first term on the right-hand side of \eqref{eq-limite}. Finally, by letting $\epsilon \to 0$ in \eqref{eq-limite}, we obtain the limit equation
\begin{equation}\label{eqlimite}
  \partial_{t}p^0 - \Delta p^0 - 2\frac{\nabla K}{K}\nabla p^0 = f(p^0) + \frac{\Delta K}{K} g(p^0),
\end{equation}
with the expressions of $f$ and $g$ given in the statement of Theorem \ref{thm-conv}.
\end{proof}

\section{Blocking in heterogeneous setting}\label{section4}
This section is devoted to the proofs of the main result on the blocking.

\subsection{No blocking for small variations}
We first consider the proof of Lemma  \ref{c small}.

\textit{Proof of Lemma \ref{c small}.}
Consider the equation.
\begin{align}\label{subonde}
  &\partial_t p - \partial_{xx} p - 2 \frac{ K^{'}}{K} \partial_x p  \\
  & = \frac{p(1-p)}{(1-s_f)p + (1-p)(1-s_h p)} \left( p \left( s_h \left( \delta d_u - \frac{ K^{''}}{K} \right) \right) - d_u(s_f + \delta - 1) \right).\notag
\end{align}

Let $ 0 < A < s_h \delta d_u $ and define
\begin{equation*}
f_A(p) = \frac{p(1-p)}{(1-s_f)p + (1-p)(1-s_h p)} \left( p A - d_u(s_f + \delta - 1) \right),
\end{equation*}
where $A$ is such that condition $ \int_{0}^{1} f_A(p) \, dp > 0$ holds, it is possible since $F(1) > 0$. Then, there exists $ c_A > 0$ and a decreasing function $\phi_A$ satisfying
\begin{equation*}
\begin{cases}
-\phi_A^{''} - c_A \phi_A^{'} = f_A(\phi_A), \\
\phi_A(-\infty) = 1,~~\phi_A(+\infty) = 0, \\
\phi_A^{'} < 0.
\end{cases}
\end{equation*}
Assume $\left|\frac{K^{''}}{K}\right| \le s_h \delta d_u - A$ and $2 \left|\frac{K^{'}}{K}\right| < c_A$. Let $p_A(t,x) = \phi_A(x - c_A t)$ with $c_A>0$. Then
\begin{equation*}
\partial_t p_A - \partial_{xx} p_A - 2 \frac{K^{'}}{K} \partial_x p_A = - 2 \frac{K^{'}}{K} \phi^{'}_A - \phi_A^{''} = \left[ c_A - 2 \frac{K^{'}}{K} \right] \phi^{'}_A + f_A(\phi_A).
\end{equation*}
Since $2 \left| \frac{K^{'}}{K} \right| < c_A$ and $\phi_A^{'} < 0$, it follows that
\begin{equation*}
\partial_t p_A - \partial_{xx} p_A - 2 \frac{K^{'}}{K} \partial_x p_A < f_A(\phi_A) = f_A(p_A).
\end{equation*}
Finally, using $\left|\frac{K^{''}}{K}\right| \le s_h \delta d_u - A$, we have
\begin{align*}
  &\partial_t p_A - \partial_{xx} p_A - 2 \frac{K^{'}}{K} \partial_x p_A \\
  &< \frac{p_A(1-p_A)}{(1-s_f)p_A + (1-p_A)(1-s_h p_A)} \left( p_A \left( s_h \left( \delta d_u - \frac{ K^{''}}{K} \right) \right) - d_u(s_f + \delta - 1) \right).
\end{align*} 
Therefore, $ p_A $ is a sub-solution of equation \eqref{subonde}, which implies that there is no blocking for $ \frac{K^{'}}{K} $ and $ \frac{ K^{''}}{K} $ sufficiently small.
\qed

\subsection{Proof of Theorem \ref{thmprincipale}}~~\par
Now, we proceed to the proof of Theorem \ref{thmprincipale}, which establishes the existence of barriers.

\subsubsection{Preliminaries}

We first reformulate the problem by introducing 
\begin{equation*}
    p(0) = \beta,~~ p(L) = \alpha.
\end{equation*}
Then for $ 0 < \alpha  < \beta  < 1$, $\alpha\leq \frac{s_f}{s_h}$, we aim to find $C$ and $L > 0$ such that
\begin{equation}\label{equation5}
  \begin{cases}
-p''-2Cp' - C^{2}g(p) = f(p), \qquad \text{ on } (0,L), \\
p(0) = \beta, \hspace{0.2cm}p(L) = \alpha,\\
p{'}(0^-)-p{'}(0^+)= Cg(p(0)) = Cg(\beta),\\
p{'}(L^-)-p{'}(L^+)= -Cg(p(L)) = -Cg(\alpha),\\
\frac{1}{2}p'(0^-)^{2}+F(\beta) = F(1),~~  \frac{1}{2}p'(L^+)^{2}+F(\alpha) = 0.
\end{cases}
\end{equation}
The following Lemma states a link between equations \eqref{equationelliptique} and \eqref{equation5}.\\
\begin{lemm}\label{lemma1}
Let $C, L > 0$. If $(C,L) \in B(f)$, then there exists $(\alpha,\beta)$ such that \eqref{equation5} has a solution. Conversely, if there exist $\alpha$, $\beta$, $C$, and $L$ such that \eqref{equation5} has a solution, then this solution is also a solution of \eqref{equationelliptique}.
\end{lemm}

\begin{proof}
  Let $C$ and $L$ be positive constants. If $(C,L) \in B(f)$, then the problem (\ref{equationelliptique}) admits a solution denoted by $p$.
  Thus, the first equation of problem \eqref{equation5} is satisfied. It remains to prove the existence of $\alpha$ and $\beta$ that fulfill the conditions:
  \begin{equation*}
    p(0) = \beta,\hspace{0.2cm} p(L) = \alpha,\hspace{0.2cm} \frac{1}{2}p'(0^-)^{2} + F(\beta) = F(1), \hspace{0.2cm} \frac{1}{2}p'(L^+)^{2}+ F(\alpha) = 0.
  \end{equation*}
  Let us define $\alpha, \beta > 0$ such that $p(0) = \beta$ and $p(L) = \alpha$.
  By conservation of the quantity $\frac 12 p'(x)^2 + F(p(x))$ on $(-\infty,0)$ and on $(L,+\infty)$, we get that $\frac{1}{2}p'(0^-)^{2} + F(\beta) = F(1)$ and $\frac{1}{2}p'(L^+)^{2}+ F(\alpha) = 0$.
  Then, we establish the existence of $\alpha$ and $\beta$ that fulfill the conditions in Equation \eqref{equation5} and ensure the solvability of the problem. To reverse the statement, we can employ a similar line of reasoning.
\end{proof}

System \eqref{equation5} may be rewritten as a boundary value problem on $(0,L)$ for $0<\alpha < \beta<1$, $\alpha\leq \frac{s_f}{s_h}$~:
\begin{equation}\label{bvp}
  \begin{cases}
    -p'' - 2 C p' - C^{2}g(p) = f(p),  \qquad \text{ on } (0,L), \\
    p(0) = \beta, \qquad   p{'}(0^+) = - \sqrt{2(F(1)-F(\beta))} - Cg(\beta),\\
    p(L) = \alpha, \qquad p{'}(L^-) = - \sqrt{-2F(\alpha)} -Cg(\alpha).
  \end{cases}
\end{equation}

  \begin{lemma}\label{lem:monoton}
    Under assumption \eqref{hyp_coeff}--\eqref{s_h}, let $0<\alpha<\beta<1$, $\alpha\leq \frac{s_f}{s_h}$, and $C>0$ be such that $\sqrt{2(F(1)-F(\beta))} + Cg(\beta)>0$, then any solution $p$ of \eqref{bvp}, such that $0\leq p \leq 1$, is nonincreasing on $(0,L)$.
\end{lemma}
\begin{proof}
Let us define
\begin{equation*}
h(p):=f(p)+C^{2}g(p)=\frac{p(1-p)}{(1-s_{f})p+(1-p)(1-s_{h}p)} A(p),
\end{equation*}
where
\begin{equation*}
A(p)=(d_{u}\delta - C^2) s_{h} p - d_u (s_{f} + \delta - 1) + C^{2}s_{f}.
\end{equation*}
Clearly, the term in factor of $A$ is nonnegative for $p\in(0,1)$, then the sign of $h$ is the same as the sign of $A$ which vanishes at the point
$$
p_h := \dfrac{d_u (s_f + \delta -1) - C^2 s_f}{(d_u \delta - C^2) s_h}.
$$
Then, when \eqref{hyp_coeff}--\eqref{s_h} holds, we have that
\begin{itemize}
\item If $0\leq C^2 < \frac{d_u}{s_h-s_f}(\delta s_h -s_f-\delta+1)$ (this latter quantity is positive thanks to \eqref{s_h}), we have $p_h\in (0,1)$. Then, $h(p)<0$ for $p\in(0,p_h)$, and $h(p)>0$ for $p\in (p_h,1)$. 
\item If $C^2\in [\frac{d_u}{s_h-s_f}(\delta s_h -s_f-\delta+1),d_u \frac{s_f+\delta-1}{s_f}]$, then $h(p)<0$ for $p\in(0,1)$.
\item If $C^2 > d_u \frac{s_f+\delta-1}{s_f}$, then $p_h\in (0,1)$ and $h(p)>0$ for $p\in(0,p_h)$, and $h(p)<0$ for $p\in (p_h,1)$.
\end{itemize}

\medskip

Let us take $C$ and $\beta$ such that $\sqrt{2(F(1)-F(0))} + Cg(\beta) >0$, then we have $p'(0^+)<0$ and $p'(L^-)<0$ (recalling that $g(\alpha)\geq 0$ for $\alpha\leq \frac{s_f}{s_h}$) and $p$ is regular (at least $C^1$). By contradiction, if $p$ is not nonincreasing on $(0,L)$, then there exists a local minimum point  $x_m$ and a local maximum point $x_M>x_m$, and $p$ is nondecreasing on $(x_m,x_M)$.

By definition of the maximum point, we have $p'(x_M)=0$ and $p''(x_M)<0$. Therefore, from the first equation in \eqref{bvp}, we get
$$
h(p(x_M)) = - p''(x_M) >0.
$$
It is not possible if $C^2\in [\frac{d_u}{s_h-s_f}(\delta s_h -s_f-\delta+1),d_u \frac{s_f+\delta-1}{s_f}]$. If $C^2 > d_u \frac{s_f+\delta-1}{s_f}$, it implies that $p(x_M) <p_h$.
However, at the minimum point, we have $p'(x_m)=0$ and $p''(x_m)>0$, hence
$$
h(p(x_m)) = - p''(x_m) <0.
$$
Then, if $C^2 > d_u \frac{s_f+\delta-1}{s_f}$, it implies that $p(x_m) > p_h > p(x_M)$.
This is a contradiction with the fact that $p$ is nondecreasing on $(x_m,x_M)$.

We are left to consider the case $C^2 < \frac{d_u}{s_h-s_f}(\delta s_h -s_f-\delta+1)$.
In this situation, we have $p(x_m)\in (0,p_h)$.
We define $E(x) = \frac 12 (p'(x))^2 + H(p(x))$ for $H=F+C^2 G$ the antiderivative of $h$ which vanishes in $0$. (In particular, since $h<0$ on $(0,p_h)$, we have $H<0$ on $(0,p_h)$.)  
We compute with \eqref{bvp}
$$
E'(x) = p''(x) p'(x) + h(p(x)) p'(x) = -\textcolor{blue}{2} C (p'(x))^2 \leq 0.
$$
Then, $E$ is nonincreasing, in particular at the minimum point $x_m\in (0,L)$, we have
\begin{align*}
  E(x_m) \geq E(L)
  & = \frac 12 (\sqrt{-2F(\alpha)}+Cg(\alpha))^2 + F(\alpha)+C^2 G(\alpha)  \\
  & = C\sqrt{-2F(\alpha)}g(\alpha) + \frac{C^2}{2} g(\alpha)^2 + C^2 G(\alpha) >0.
\end{align*}
Indeed, for $\alpha\leq \frac{s_f}{s_h}$, we have $g(\alpha)\geq 0$ (which implies obviously $G(\alpha)>0$).
However,
$E(x_m) = H(p(x_m)) < 0$, since $H$ is negative on $(0,p_h)$. This is a contradiction.
\end{proof}

\subsubsection{Existence of a stationary solution}~~\par
 
  Existence of a barrier boils down to show existence of a solution to \eqref{bvp}.
  The following Lemma allows us to prove Theorem \ref{thmprincipale}.
  \begin{lemm}[Existence of solutions]\label{existence CL}
  Assume that \eqref{hyp_coeff}--\eqref{s_h} hold, and that $f$ and $g$ are given by \eqref{def:f} and \eqref{def:g} with $F(1) > 0$.
  Then, for $\beta=\dfrac{s_f}{s_h}$ there exists $C^*>0$ such that for any $C>C^*$ there exist $\alpha \in (\displaystyle\argmax_{(0,1)} g,\dfrac{s_f}{s_h})$ and $L>0$, such that \eqref{bvp} admits a solution, which is decreasing.
\end{lemm}

\begin{proof} 
  The problem being rewritten as the boundary value problem \eqref{bvp}, we will use a shooting argument. We recall that $g\big(\frac{s_f}{s_h}\big)=0$, and $f<0$ on $(0,\frac{s_f}{s_h})$ and $g>0$ on $(0,\frac{s_f}{s_h})$. Then, in \eqref{bvp} we have $\frac 12 p^{'}(0^+)^2 = F(1)-F(\frac{s_f}{s_h})$.

  Let $p_{\beta}$ be the unique solution of the Cauchy problem on $(0,L)$~:
  \begin{equation}\label{equation6}
    \begin{cases}
      -p_{\beta}^{''} - 2 C p_{\beta}^{'} - C^{2}g(p_{\beta}) = f(p_{\beta}), \qquad \text{ on } (0,L), \\
      p_{\beta}(0) = \frac{s_f}{s_h}, \qquad 
      p_{\beta}^{'}(0^+) = -\sqrt{2(F(1) - F(\frac{s_f}{s_h}))}.
    \end{cases}
  \end{equation}
  Since we are looking for a decreasing solution, we may invert it and define $X_{\beta} :\mathclose{]}p_{\beta}(L), p_{\beta}(0)\mathclose{[}\to \mathclose{]}0,L\mathclose{[}$ such that $p_{\beta}(X_{\beta}(p))=p$.
  We introduce the function
  \begin{align*}
    w_{\beta}(p) = \frac{1}{2}p_{\beta}^{'}(X_{\beta}(p))^{2}.
  \end{align*}
  Then, we have, for $p\in (0,\frac{s_f}{s_h})$
  \begin{equation}\label{equation7}
    \begin{cases}
      \displaystyle \frac{dw_{\beta}}{dp} = 2 C\sqrt{2 w_{\beta}(p)} - f(p) - C^2 g(p),\\[1em]
      w_{\beta}(\frac{s_f}{s_h}) = F(1)-F(\frac{s_f}{s_h}).
    \end{cases}
  \end{equation}
  The question is to prove that there exists $\alpha\in (\displaystyle\argmax_{(0,1)} g,\dfrac{s_f}{s_h})$ such that
  $$
  w_\beta(\alpha) = \dfrac{1}{2} \big(\sqrt{-2F(\alpha)} + Cg(\alpha)\big)^2.
  $$
  
\textbf{Bound from above~:}
To ease the study, we introduce the function $z(\tau) = w_\beta(\frac{s_f}{s_h}-\tau)$. Then, $z$ satisfies, for $\tau>0$,
  \begin{equation}\label{equation_z}
    \begin{cases}
      \displaystyle z'(\tau) = - 2 C\sqrt{2 z(\tau)} + f(\frac{s_f}{s_h}-\tau) + C^2 g(\frac{s_f}{s_h}-\tau),\\[1em]
      z(0) = F(1)-F(\frac{s_f}{s_h}).
    \end{cases}
  \end{equation}
  Let us define
  $$
  \widetilde{z}(\tau) = - F(\frac{s_f}{s_h} - \tau) + \frac{C^2}{8} g_\infty^2,
  $$
  where $F$ is the antiderivative of $f$ vanishing in $0$, and $g_\infty = \max_{(0,\frac{s_f}{s_h})} g$. Since $f<0$ on $(0,\frac{s_f}{s_h})$, we have
  $$
  \widetilde{z}(\tau) \geq \frac{C^2}{8} g_\infty^2, \quad \text{ or equivalently }\quad
  C^2 g_\infty - 2C\sqrt{2\widetilde{z}(\tau)} \leq 0.
  $$
  Thus, we compute
  $$
  \widetilde{z}^{'}(\tau) = f(\frac{s_f}{s_h}-\tau) \geq C^2 g_\infty - 2C\sqrt{2\widetilde{z}(t)} + f(\frac{s_f}{s_h}-\tau).
  $$
  Moreover, we have
  $$
  \widetilde{z}(0) = - F(\frac{s_f}{s_h}) + \frac{C^2}{8} g_\infty^2.
  $$
  Then, for $C^2 \geq \frac{8F(1)}{g_\infty^2}$, we deduce that $\widetilde{z}(0)\geq z(0)$, thus $\widetilde{z}$ is a super-solution for the problem \eqref{equation_z}. As a consequence, for any $\alpha < \frac{s_f}{s_h}$, we have
  \begin{equation}\label{bornesup_w}
    w_\beta(\alpha) = z\big(\frac{s_f}{s_h}-\alpha\big) \leq \widetilde{z}\big(\frac{s_f}{s_h}-\alpha\big) = - F(\alpha) + \frac{C^2}{8} g_\infty^2.
  \end{equation}
  Therefore, if we take $\alpha = \argmax g \in (0,\frac{s_f}{s_h})$, we have $g_\infty = g(\alpha)$ and from \eqref{bornesup_w} we get
  \begin{equation}
    \label{ineq1'}
    w_\beta(\alpha) \leq - F(\alpha) + \frac{C^2}{8} g(\alpha)^2 \leq \frac {1}{2} \big(\sqrt{-2F(\alpha)}+Cg(\alpha)\big)^2.
  \end{equation}

  Moreover, when $\alpha\to \frac{s_f}{s_h}$, we have that 
  $\frac 12 \big(\sqrt{-2F(\alpha)} + Cg(\alpha)\big)^2 \to -F(\frac{s_f}{s_h})$ and $w_\beta(\alpha) \to w_\beta(\frac{s_f}{s_h}) = F(1)-F(\frac{s_f}{s_h})$ from \eqref{equation7}.
  Since $F(1)>0$, we deduce that, by continuity, for $\alpha$ in the neighborhood of $\frac{s_f}{s_h}$, we have
  \begin{equation}
    \label{ineq2}
    \frac {1}{2} \big(\sqrt{-2F(\alpha)} + Cg(\alpha)\big)^2 < w_\beta(\alpha).  
  \end{equation}

  From \eqref{ineq1'} and \eqref{ineq2}, we deduce, by continuity, that for $C^2\geq \frac{\textcolor{blue}{8}F(1)}{g_\infty^2}$, there exists $\alpha\in(\argmax g,\frac{s_f}{s_h})$ such that
  $$
  \frac{1}{2} \big(\sqrt{-2F(\alpha)} + Cg(\alpha)\big)^2 = w_\beta(\alpha).  
  $$
  Then, we have constructed a solution of the boundary value problem \eqref{bvp}.
  Moreover, from Lemma \ref{lem:monoton}, we deduce that $p_\beta$ is nonincreasing.
  Finally, to recover the value of $L$ it suffices to note that
  \begin{equation*}
    L = \int_\beta^\alpha X_\beta'(p) \, dp.
  \end{equation*}
  To compute it from $w_{\beta}$, we observe that
  \begin{equation*}
    X_\beta'(p) = \frac{1}{p_\beta'(X_\beta(p))}.
  \end{equation*}
  Hence, we have
  \begin{equation*}
    L = \int_\alpha^\beta \frac{1}{\sqrt{2 w_{\beta}(p)}} \, dp.
  \end{equation*}
\end{proof}
\begin{remkm}
In the proof, we have established that we can take actually $C^* \leq \frac{2\sqrt{2F(1)}}{g_\infty}$.
\end{remkm}

\begin{lemm}
  Let us assume that there exists a $(C_1,L_1)$-barrier, then for any $L_2>L_1$ there exists a $(C_2,L_2)$-barrier for $C_2=C_1\dfrac{L_1}{L_2}$.
\end{lemm}
\begin{proof}
  Let $p_1$ be a $(C_1,L_1)$-barrier, then there exists $0<\alpha<\beta\leq \frac{s_f}{s_h}$ such that it is a solution of \eqref{bvp}. Let $L_2>L_1$ and define $p_2(x) = p_1(\frac{L_1}{L_2}x)$ for $x\in (0,L_2)$ and $C_2=C_1\dfrac{L_1}{L_2}$. We compute, for $x\in (0,L_2)$ and $y=\frac{L_1}{L_2}x \in (0,L_1)$, using the first equation in \eqref{bvp},
  \begin{align*}
    -p_2^{''}(x) -2C_2 p_2^{'}(x)
    & = \left(\frac{L_1}{L_2}\right)^2 \left(-p_1^{''}(y) -2C_1p_1'(y)\right)  \\
    & = \left(\frac{L_1}{L_2}\right)^2 \left(C_1^2 g(p_1(y)) + f(p_1(y))\right) \\
    & = (C_2)^2 g(p_2(x)) + \left(\frac{L_1}{L_2}\right)^2 f(p_2(y)) \\
    & \geq (C_2)^2 g(p_2(x)) + f(p_2(y)), 
  \end{align*}
since we recall that $f<0$ on $(0,\frac{s_f}{s_h})$ and $L_1<L_2$.
And, from the second equation in \eqref{bvp},
\begin{align*}
  p_2^{'}(0^+)
  & = \frac{L_1}{L_2} p_1^{'}(0^+) = - \frac{L_1}{L_2}\sqrt{2(F(1)-F(\beta))} - C_2 g(\beta)   \\
  & \geq - \sqrt{2(F(1)-F(\beta))} - C_2 g(\beta)
\end{align*}
Similarly,
\begin{align*}
  p_2^{'}(L^-)
  & \geq -\sqrt{- 2 F(\alpha)} - C_2 g(\alpha).
\end{align*}
Hence, $p_2$ is a super-solution for \eqref{bvp} with $C=C_2$ and $L=L_2$. Hence there exists a $(C_2,L_2)$-barrier.
\end{proof}

\medskip
Having established the existence of a barrier, we now proceed with the proof of the final main result, which is Proposition \ref{principe de comparaison}.
\begin{proof}[Proof of Proposition \ref{principe de comparaison}]~~\par
\vspace{0.2cm}
The result follows directly from the maximum principle, given that $ \bar{P} $ is a barrier and $ p^{ini} \leq \bar{P} $.  
Therefore, the solution remains bounded by $ \bar{P} $ for all $ t \geq 0 $, implying that the spreading stops.
\end{proof}

The results of Theorem \ref{thmprincipale} and Proposition \ref{principe de comparaison} conclude that if the initial positive solution is smaller than the stationary solution, it can never reach or exceed the stationary solution. Indeed, the stationary solution acts as a barrier to propagation. The term 'blocking' is therefore justified, as it describes the situation where the solution consistently remains below the stationary solution, thus confirming the phenomenon of propagation blocking.
\newpage
\vspace{0.2cm}

\section{ Numerical simulations}\label{section5}~~\par
We performed numerical simulations of the mathematical equation \eqref{notre systeme} to illustrate our theoretical results. We used the dimensionless parameter values listed in Table~\ref{tab:parameters} below, which were estimated based on the conditions identified in the theoretical analysis.
\begin {table}[!ht] 
\begin{center}
\resizebox{\textwidth}{!}{%
\begin{tabular}{||c c c ||} 
 \hline
 Parameters & Interpretations &Values \\ [0.5ex] 
 \hline\hline
 $s_f$ & a dimensionless parameter that reflects
the reduced fecundity of infected mosquitoes &$0.015$ \\ 
 \hline
 $F_u$ & Fecondity & $1$ \\
 \hline
 $d_u$  & Mortality rate of non-infected mosquitoes & $0.30$\\
 \hline
 $d_i$  & Mortality rate of infected mosquitoes & $0.315$\\
 \hline
 $\delta$ & The ratio $\frac{d_i}{d_u}$ & $1.05$ \\
 \hline
 $s_h$ & Cytoplasmic incompatibility parameter & $0.7$ \\ [1ex] 
\hline
 $D$ & The diffusion constant & $0.2$\\
 \hline
\end{tabular}}
\vspace{0.5cm}
\end{center}
\caption {Specified parameter values for the model. The values used in the simulations are inspired by those presented in \cite{duprez2021optimization}} 

\label{tab:parameters} 
\end {table}\\
We performed numerical simulations of the system in a one-dimensional space. Two different final times were considered in order to illustrate the propagation and the blocking regimes: $T = 300$ days for the propagation case and $T = 400$ days for the blocking case. The numerical domain used was the interval $[-L, L]$, with $L=20$, representing a total width of 40 units. This domain was discretized using a semi-implicit finite difference scheme with 2000 spatial points, resulting in a spatial step size of 
$dx = \frac{2L}{2000}$. A time step of $dt = 0.1$. The implicit scheme efficiently handled the complex terms in the system.

Neumann boundary conditions were applied at the edges of the domain, ensuring consistent behavior at the boundaries.

In these simulations, the function $K(x)$, which directly influences the propagation dynamics, was defined as follows:
\begin{equation*}\label{equation pour K1}
K(x) = \begin{cases}
A, &\text{for } x < 0,\\
    Ae^{Cx},& \text{for } x \in [0, \frac{L}{4}],\\
Ae^{CL/4},&\text{for } x > \frac{L}{4}.
\end{cases}
\end{equation*}
The constant $C$, which appears in the exponential expression of $K(x)$, plays a crucial role in determining the propagation dynamics. As proved in Lemma \ref{c small}, small values of $C$ allow continuous propagation. For instance, with $C = 10^{-6}$, the results illustrate uninterrupted propagation, as shown in Figures~\ref{fig:infected} and \ref{fig:non_infected}.
\begin{figure}[!ht]
\begin{minipage}[ht]{0.45\linewidth}
  \centering  
  \includegraphics[width=\textwidth]{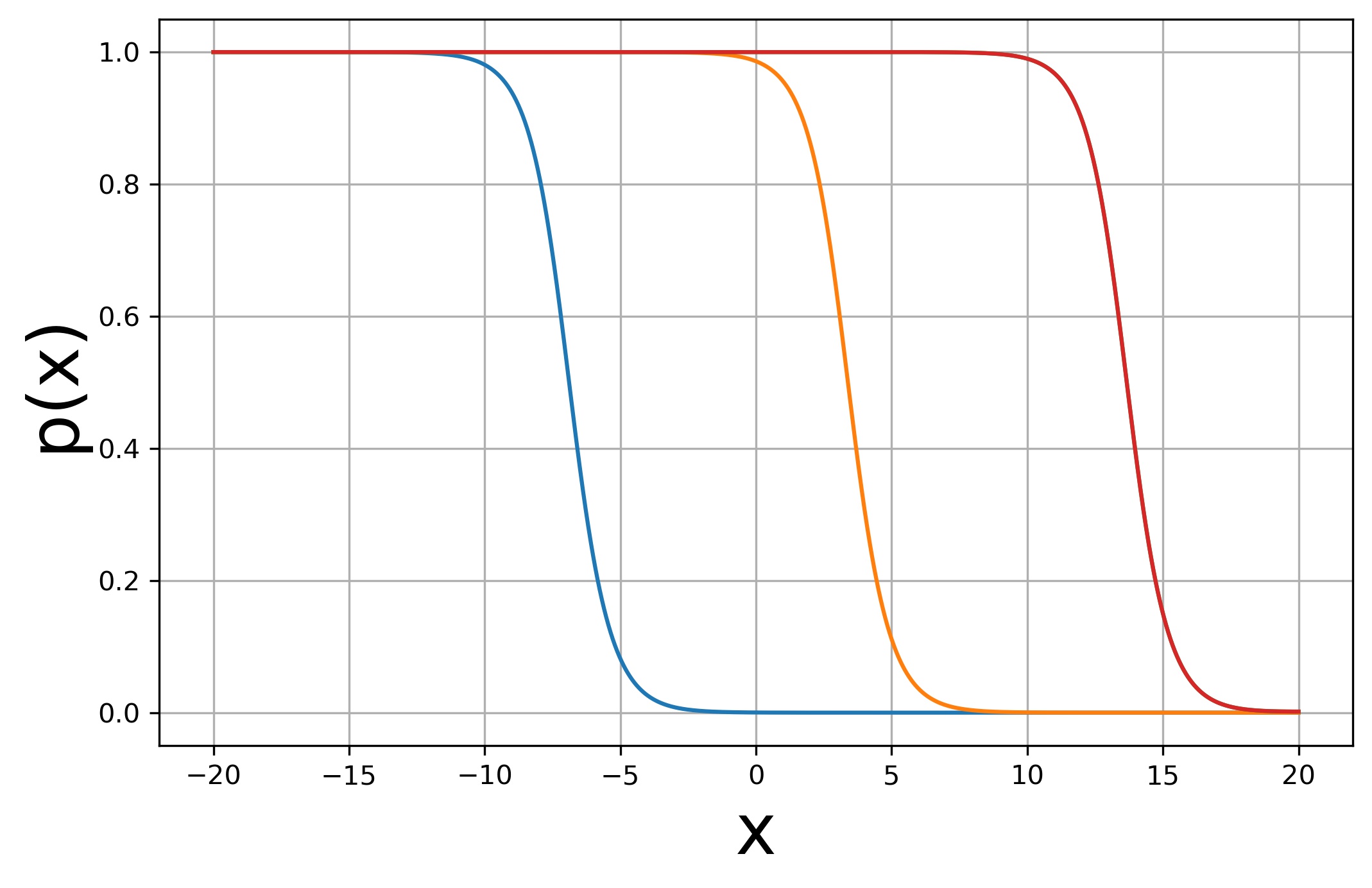}
 \caption{Propagation of the Proportion of Infected Mosquitoes at Constant Speed Over Time.} \label{fig:infected}
\end{minipage} 
  \hspace{1cm} 
\begin{minipage}[ht]{0.45\linewidth}
\centering 
  \includegraphics[width=\textwidth]{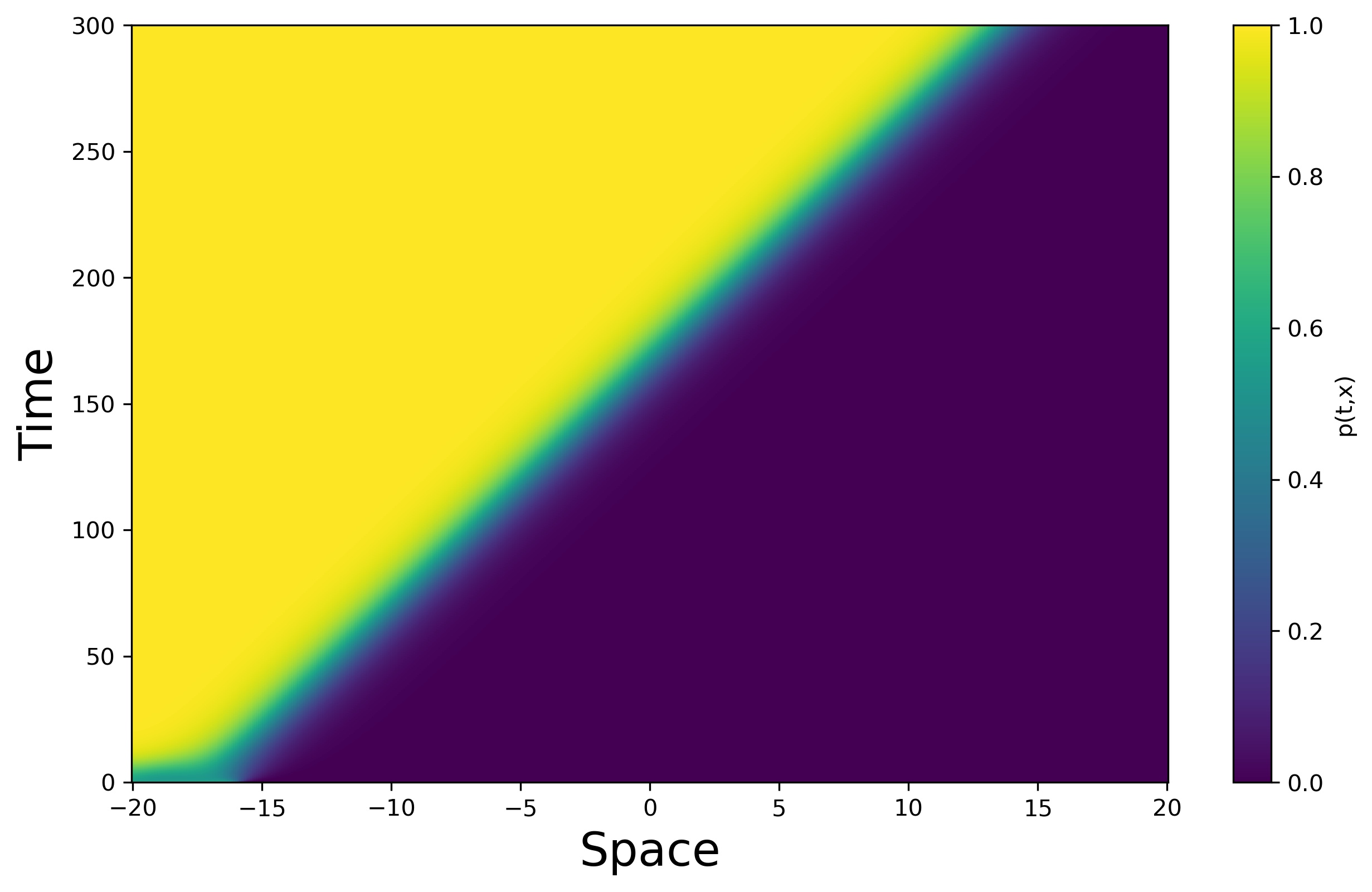}
  \caption{Spatio-Temporal Dynamics of the Proportion of Infected Mosquitoes.}\label{fig:non_infected}
\end{minipage} 
\end{figure}\\
In a second simulation, we studied a propagation blocking scenario by extending the simulation duration to $T = 400$ days. A moderately high value of $C$, sufficiently large to cause blocking, was used. Specifically, we set $C = 6$. 

The observed blocking can be explained by the exponential nature of the function $K(x)$, where higher values of $C$ effectively limit propagation. The results illustrate the blocking of propagation, as shown in Figures~\ref{fig:infectedblocage} and \ref{fig:non_infectedblocage}.

\begin{figure}[ht]
\begin{minipage}[h]{0.45\linewidth}
  \centering  
  \includegraphics[width=\textwidth]{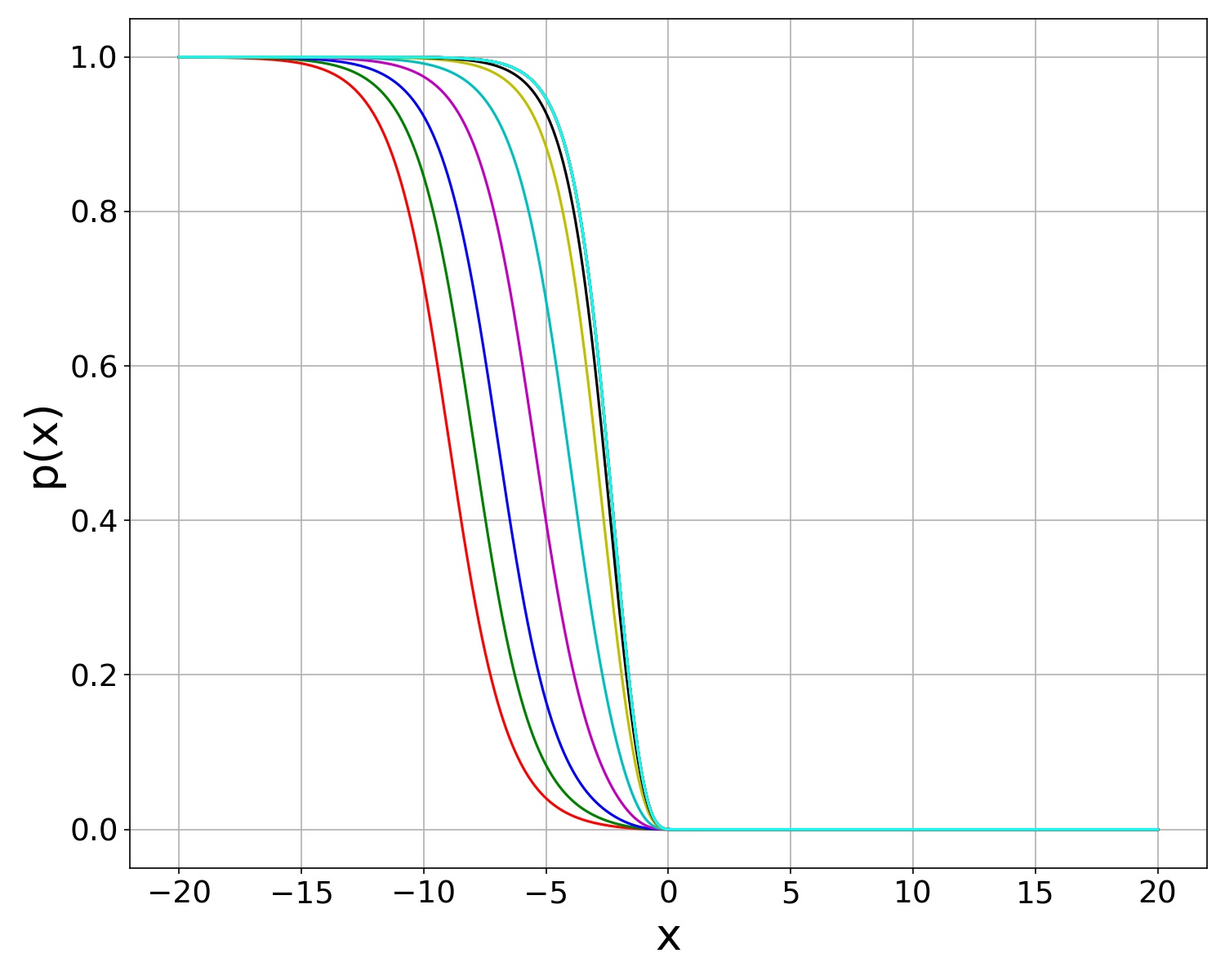}
 \caption{Blocking of the Propagation of the Proportion of Infected Mosquitoes.} \label{fig:infectedblocage}
\end{minipage} 
  \hspace{1cm} 
\begin{minipage}[ht]{0.45\linewidth}
\centering 
  \includegraphics[width=\textwidth]{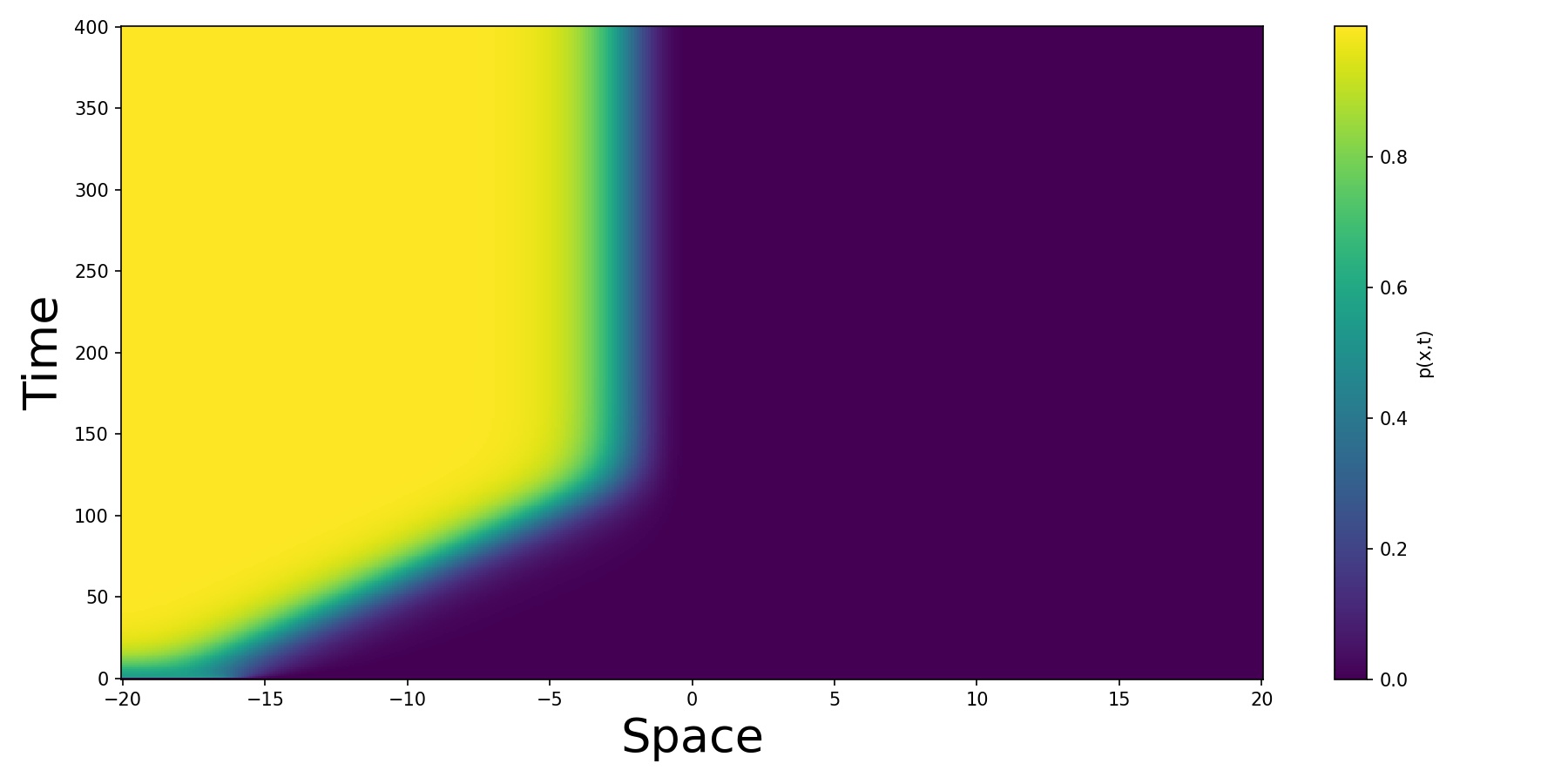}
  \caption{Spatiotemporal Dynamics of the Proportion of Infected Mosquitoes Spread and Blockage.}\label{fig:non_infectedblocage}
\end{minipage} 
\end{figure}
\newpage

\section{Discussion and Perspectives}\label{section6}~~\par
In this paper, we rigorously established the convergence of a reaction-diffusion system to a reduced model describing the frequency of a variant. This provides a mathematical justification for the use of simplified models, particularly bistable equations. We then used this reduced model to analyze the impact of environmental heterogeneity on the spread of Wolbachia.

Several research directions can be considered to extend this work:
\begin{itemize}
    \item[$\bullet$] Extension to a two-population system: The idea is to study a system of two coupled reaction-diffusion equations representing the dynamics of infected and uninfected populations. Although the analysis becomes more complex, the monotonicity properties of such systems may allow for a rigorous study.

\item[$\bullet$] Extension to higher spatial dimensions: So far, the study has been limited to a one-dimensional space. Considering models in two or three dimensions would allow for a more realistic representation of environments and could capture more complex spatial effects.\end{itemize}

\bibliographystyle{abbrv}
\bibliography{bibliography}
\end{document}